\documentclass[11pt]{amsart}
\usepackage[T1]{fontenc}
\usepackage[english]{babel}
\usepackage{amscd,amsmath,amsthm,amssymb,graphics}
\usepackage{lmodern,pst-node}
\usepackage{pstcol,pst-plot,pst-3d}
\usepackage{multicol}
\usepackage{epic,eepic}
\usepackage{amsfonts,amssymb,amscd,amsmath,enumitem,verbatim}
\psset{unit=0.7cm,linewidth=0.8pt,arrowsize=2.5pt 4}

\newpsstyle{fatline}{linewidth=1.5pt}
\newpsstyle{fyp}{fillstyle=solid,fillcolor=verylight}
\definecolor{verylight}{gray}{0.97}
\definecolor{light}{gray}{0.9}
\definecolor{medium}{gray}{0.85}
\definecolor{dark}{gray}{0.6}

\def\frk{\frak}               

\def\Phi{{\frk n}}
\def\Phi{{\frk N}}
\def\opn#1#2{\def#1{\operatorname{#2}}} 
\opn\chara{char} \opn\length{\ell} \opn\pd{pd} \opn\rk{rk}
\opn\projdim{proj\,dim} \opn\injdim{inj\,dim} \opn\rank{rank}
\opn\depth{depth} \opn\grade{grade} \opn\height{height}
\opn\embdim{emb\,dim} \opn\codim{codim}

\opn\Tr{Tr} \opn\bigrank{big\,rank}
\opn\superheight{superheight}\opn\lcm{lcm}
\opn\trdeg{tr\,deg}
\opn\reg{reg} \opn\lreg{lreg} \opn\ini{in} \opn\lpd{lpd}
\opn\size{size}\opn\bigsize{bigsize}
\opn\cosize{cosize}\opn\bigcosize{bigcosize}
\opn\sdepth{sdepth}\opn\sreg{sreg}
\opn\link{link}\opn\fdepth{fdepth}
\opn\deg{deg}
\opn\max{max}
\opn\indeg{indeg}
\opn\min{min}
\opn\psln{psln}
\opn\div{div} \opn\Div{Div} \opn\cl{cl} \opn\Cl{Cl}
\let\epsilon\varepsilon
\let\phi=\varphi
\let\kappa=\varkappa

\opn\Spec{Spec} \opn\Supp{Supp} \opn\supp{supp} \opn\Sing{Sing}
\opn\Ass{Ass} \opn\Min{Min}\opn\Mon{Mon} \opn\dstab{dstab} \opn\astab{astab}
\opn\Syz{Syz}
\opn\Ann{Ann} \opn\Rad{Rad} \opn\Soc{Soc}
\opn\Im{Im} \opn\Ker{Ker} \opn\Coker{Coker} \opn\Am{Am}
\opn\Hom{Hom} \opn\Tor{Tor} \opn\Ext{Ext} \opn\End{End}
\opn\Aut{Aut} \opn\id{id}

\opn\nat{nat}
\opn\pff{pf}
\opn\Pf{Pf} \opn\GL{GL} \opn\SL{SL} \opn\mod{mod} \opn\ord{ord}
\opn\Gin{Gin} \opn\Hilb{Hilb}\opn\sort{sort}
\opn\initial{init}
\opn\ende{end}
\opn\height{height}
\opn\depth{depth}
\opn\type{type}
\opn\ldim{ldim}
\opn\lk{lk}
\opn\del{del}

\opn\aff{aff} \opn\con{conv} \opn\relint{relint} \opn\st{st}
\opn\lk{lk} \opn\cn{cn} \opn\core{core} \opn\vol{vol}
\opn\link{link} \opn\star{star}\opn\lex{lex}
\opn\gr{gr}

\def\pot#1#2{#1[\kern-0.28ex[#2]\kern-0.28ex]}

\opn\dirlim{\underrightarrow{\lim}}
\opn\inivlim{\underleftarrow{\lim}}
\def\Implies{\ifmmode\Longrightarrow \else
        \unskip${}\Longrightarrow{}$\ignorespaces\fi}
\def\implies{\ifmmode\Rightarrow \else
        \unskip${}\Rightarrow{}$\ignorespaces\fi}
\def\iff{\ifmmode\Longleftrightarrow \else
        \unskip${}\Longleftrightarrow{}$\ignorespaces\fi}

\let\:=\colon
 \theoremstyle{plain}
\newtheorem{Theorem}{Theorem}[section]
 \newtheorem{Lemma}[Theorem]{Lemma}
 \newtheorem{Corollary}[Theorem]{Corollary}
 \newtheorem{Proposition}[Theorem]{Proposition}

 \newtheorem{Question}[Theorem]{Question}
 \theoremstyle{definition}
 \newtheorem{Definition}[Theorem]{Definition}

 \newtheorem{Example}[Theorem]{Example}

\let\epsilon\varepsilon
\let\kappa=\varkappa
\def\qed{\ifhmode\textqed\fi
      \ifmmode\ifinner\quad\qedsymbol\else\dispqed\fi\fi}
\def\textqed{\unskip\nobreak\penalty50
       \hskip2em\hbox{}\nobreak\hfil\qedsymbol
       \parfillskip=0pt \finalhyphendemerits=0}
\def\dispqed{\rlap{\qquad\qedsymbol}}

\opn\dis{dis}
\def\pnt{{\raise0.5mm\hbox{\large\bf.}}}

\opn\Lex{Lex}

\begin{document}

\author[Mafi and Rasul Qadir]{ Amir Mafi and Rando Rasul Qadir}
\title{Betti numbers and almost complete intersection monomial ideals}

\address{Amir Mafi, Department of Mathematics, University Of Kurdistan, P.O. Box: 416, Sanandaj, Iran.}
\email{a\_mafi@ipm.ir}
\address{Rando Rasul Qadir, Department of Mathematics, University of Kurdistan, P.O. Box: 416, Sanandaj,
Iran.}
\email{rando.qadir@univsul.edu.iq.}

\begin{abstract}
Let $R=K[x_1,\ldots, x_n]$ be the polynomial ring in $n$ variables over a field $K$ and let $I$ be a monomial ideal of $R$. 
In this paper, we present an explicit formula for the Betti numbers of almost complete intersection monomial ideals,
which enables a rapid construction of their minimal free resolutions. In addition, we characterize the Cohen-Macaulayness of these ideals and also
we show the same result for dominant monomial ideals.
\end{abstract}
\subjclass[2020]{13C14, 13H10, 13D02, 13F55}
\keywords{Cohen-Macaulay, almost complete intersection, monomial ideal, Betti numbers.}

\maketitle
\section*{Introduction}
Throughout this paper, let $R=K[x_1,\ldots,x_n]$ denote a polynomial ring in $n$ variables over a field $K$, and let $I$ be a monomial ideal of $R$. We denote by $G(I)=\{u_1,\ldots,u_q\}$ the unique minimal set of monomial generators of $I$. For a monomial $u=x_1^{a_1} \ldots x_n^{a_n}$, we define its support as $\supp(u) = \{x_i | a_i > 0 \}$.

Suppose $\Delta$ is a simplicial complex, and let $I=I_{\Delta}$ be its Stanley-Reisner ideal generated by all squarefree monomials $x_{i_1}x_{i_2}\ldots x_{i_r}$ such that $\{i_1,\ldots,i_r\}$ is not a face of $\Delta$. The Stanley-Reisner ring is then given by $K[\Delta]=R/I$. The squarefree monomial ideal $I$ with a minimal prime decomposition $I = \cap_{i=1}^{t} \frak{p_{i}}$, where each  $\frak{p_{i}}$ is a monomial prime ideal. Then the {\it Alexander dual} $I^{\vee}$ of $I$ is generated by monomials $u_i = \prod_{x_{j} \in G(\frak{p_{i}})} x_{j}$, where $G(\frak{p_{i}})$ is the minimal set of variables generating $\frak{p_{i}}$.

In the context of combinatorial topology, Dress \cite{D} established that a simplicial complex  $K[\Delta]$ is shellable in the non-pure case if and only if $K[\Delta]$
 is clean. Also, Herzog, Hibi, and Zheng \cite{HHZ1} proved that  $K[\Delta]$ is shellable in the non-pure case if and only if the Alexander dual ideal $I^{\vee}$ has linear quotients.

The minimal free resolution of the cyclic module $R/I$ is considered as:

$$0\longrightarrow R^{\beta_{n}}\longrightarrow\ldots\longrightarrow R^{\beta_{1}}\longrightarrow R^{\beta_{0}} \longrightarrow R/I\longrightarrow 0,$$

where the integers $\beta_{i}$ are the {\it Betti numbers} of $R/I$. The Taylor resolution \cite[p. 439]{E} provides a construction for a free resolution of $R/I$. Alesandroni \cite[Theorem 4.4]{A} characterized the monomial ideals for which this resolution is minimal. Generally, the Taylor resolution is far from minimal, but it provides an upper bound for the Betti numbers of 
$R/I$ as follows: $\beta_{i}(R/I)\leq{q\choose i}$, for $0\leq i\leq q$. In \cite[Remark 1.6]{HHMY}, it is shown that if $\beta_{i}(R/I)={q\choose i}$ for some $0\leq i\leq q$, then $\beta_{j}(R/I)={q\choose j}$ for all $j\leq i$. Furthermore, Brun and R\"omer \cite[Corollary 4.1]{BR} established a lower bound for the Betti numbers as follows:  $\beta_{i}(R/I)\geq{p\choose i}$ for $0\leq i\leq p$, where $p$ denotes the projective dimension of $R/I$.

In this paper, we present an explicit formula for the Betti numbers of almost complete intersection monomial ideals. Furthermore, we show that an almost complete intersection monomial ideal is Cohen-Macaulay if and only if it is unmixed. Furthermore, we show the same result for dominant monomial ideals.
For any unexplained notion or terminology, we refer the reader to \cite{HH} and \cite{E}. Several explicit examples were  performed with help of the computer algebra system Macaulay 2 \cite{G}.

\section{Main Results}

We begin this section by recalling that a monomial ideal $I$ of $R$ is called an {\it almost complete intersection} (resp. {\it a complete intersection}) if $\height(I)=\mid G(I)\mid+1$ (resp. $\height(I)=\mid G(I)\mid)$.

Kimura, Terai and Yoshida in \cite[Theorem 4.4]{KTY} proved the following interesting result:

\begin{Theorem}\label{T0}

 If $I$ is an almost complete intersection squarefree monomial ideal of $\height(I)=q\geq 2$, then $I$ has one of the following forms:
\begin{enumerate}
\item[(i)] $I=(u_1v_1,u_2v_2,\ldots,u_rv_r,u_{r+1},\ldots,u_q,v_1v_2\ldots v_r)$;
\item[(ii)] $I=(u_1v_1,u_2v_2,\ldots,u_rv_r,u_{r+1},\ldots,u_q,u_{q+1}v_1v_2\ldots v_r)$; 
\item[(iii)] $I=(v_1v_2,v_1v_3,v_2v_3,u_4,\ldots,u_{q+1})$;
\item[(iv)] $I=(u_1v_1v_2,v_1v_3,v_2v_3,u_4,\ldots,u_{q+1})$;
\item[(v)]  $I=(u_1v_1v_2,u_2v_1v_3,v_2v_3,u_4,\ldots,u_{q+1})$;
\item[(vi)] $I=(u_1v_1v_2,u_2v_1v_3,u_3v_2v_3,u_4,\ldots,u_{q+1})$;
\end{enumerate}	
where $u_1,u_2,\ldots, v_1,v_2,\ldots$ are non-trivial squarefree monomials no two of which have common factors, and $r$ is integer with $2\leq r\leq q$. 
\end{Theorem}

Now, we recall the following definition which is appeared in \cite[Definiation 4.1, 6.1]{A}.

\begin{Definition}
\begin{enumerate}
\item[(a)] A monomial ideal $I$ of $R$ is called a dominant ideal if, for every monomial element of the minimal generating set $G(I)$, there exists a variable in that monomial whose degree is greater than the degree of the corresponding variable in all other monomial elements of $G(I)$.
\item[(b)] A monomial ideal $I$ of $R$ is called a semidominant ideal if $G(I)$ has at most one nondominant monomial element.  
\end{enumerate}	
\end{Definition}
According to \cite[Theorem 4.4]{A}, a monomial ideal $I$ is dominant if and only if the Taylor resolution of $I$ is minimal; equivalently, this occurs if and only if the projective dimension\\ $\pd(R/I)=\mid G(I)\mid$.
Moreover, every dominant monomial ideal is semidominant. In addition, if $I$ is semidominant, then it is Scarf, meaning that the Scarf simplicial complex provides a minimal free resolution of 
$I$.

\begin{Example}\label{E1} Let $I$ be a monomial ideal of $R$. Then the following conditions hold:
	\begin{enumerate}
\item[(i)] if $I$ is a complete intersection monomial ideal, or 
\item[(ii)] if $I$ is an universal lexsegment ideal (see \cite[Theorem 2.6]{HHMY}),
then $I$ is dominant.
\item[(iii)] If $I$ is an almost complete intersection monomial ideal, then $I$ is semidominant,
\end{enumerate}	
\end{Example}

Note that in Theorem \ref{T0} the almost complete intersections in cases $(ii)$ and $(vi)$ are dominant ideals, while the remaining cases involve semidominant ideals.\\

The following lemma is known, but we cannot find an appropriate reference; therefore, we will provide an easy proof for the readers.
\begin{Lemma}\label{L1}
Let $u=u_1,\ldots,u_q$ be monomial elements in $R$. Then $u$ is a regular sequence if and only if $\supp(u_i)\cap\supp(u_j)=\emptyset$ for all $1\leq i\neq j\leq q$. In particular, 
$I=(u_1,\ldots,u_q)$ is complete intersection if and only if $\supp(u_i)\cap\supp(u_j)=\emptyset$ for all $1\leq i\neq j\leq q$.
\end{Lemma}

\begin{proof}
Since $(u_1,\ldots,u_{i-1}):u_i=\sum_{j=1}^{i-1}(u_j:u_i)$ for $1\leq i\leq q$, we deduce that 
$(u_1,\ldots,u_{i-1}):u_i=(u_1,\ldots,u_{i-1})$ for all $1\leq i\leq q$ if and only if $\supp(u_i)\cap\supp(u_j)=\emptyset$ for all $1\leq i\neq j\leq q$. This completes the proof.
\end{proof}

Let $I$ be a monomial ideal of $R$. The cyclic module $R/I$ is defined {\it clean} if there exists a chain of monomial ideals \[\mathcal{F}: I=I_0\subset I_1\subset I_2\subset\ldots\subset I_{r-1}\subset I_r=R\] such that $I_{i}/I_{i-1}\cong R/{\frak{p}_i}$ with $\frak{p}_i$ is a minimal prime ideal of $I$. In other words, for $i=1,\ldots,r$, there exists a monomial element $u_{i}\in I_{i}$ such that $I_{i}=(I_{i-1},u_{i})$ and $\frak{p}_i=(I_{i-1}:u_{i})$. This filtration is called a {\it monomial prime filtration} of $R/I$. The set $\Supp(\mathcal{F})=\{\frak{p}_1,\ldots,\frak{p}_r\}$ is called the support of the prime filtration $\mathcal{F}$. It is known that $\Ass(I)\subseteq\Supp(\mathcal{F})\subseteq V(I)$. By applying \cite{D}, for a clean filtration $\mathcal{F}$ of $R/I$, one has $\min\Ass(I)=\Ass(I)=\Supp(\mathcal{F})$. Herzog and Popescu in \cite{HP} defined that the module $R/I$ is {\it pretty clean} when there is a prime filtration 
$\mathcal{F}$ with the following property: if $\frak{p}_i\subset\frak{p}_j$, then $j<i$. It is clear that the clean modules are pretty clean and if $I$ is a squarefree monomial ideal, then $R/I$ is pretty clean if and only if $R/I$ is clean, see \cite[Corollary 3.5]{HP}. However, this equivalence for any monomial ideal is not true, see \cite[Example 3.6]{HP}.  In addition, Herzog and Popescu \cite[Corollary 4.3]{HP} showed that if $R/I$ is pretty clean, then $R/I$ is sequentially Cohen-Macaulay.

The following definition observed in \cite[Definition 1.1]{MM}.
\begin{Definition}
A monomial ideal $I$ of $R$ is called {\it weakly polymatroidal} if for every two monomial elements $u=x^{a_1}\ldots x_n^{a_n}$ and $v=x_1^{b_1}\ldots x_n^{b_n}$ in $G(I)$ such that $a_1=b_1,\ldots, a_{t-1}=b_{t-1}$ and $a_t>b_t$ for some $t$, then there exists $j>t$ such that $x_t(v/x_j)\in I$.
\end{Definition}

In the outlined below, we consider that a monomial ideal $I=(u_1,\ldots,u_q,v)$ is an almost complete intersection ideal if $(u_1,\ldots,u_q)$ is a complete intersection monomial ideal.
According to \cite[Theorem 1.3]{MM} weakly polymatroidal ideals has linear quotients. Therefore, by applying \cite{D, HHZ1} if $I^{\vee}$ is squarefree weakly polymatroidal ideal, then $R/I$ is clean.
The following proposition, which appeared in \cite[Lemma 2.4 and Theorem 2.5]{BDS}, is proved here in a straightforward manner.

\begin{Proposition}\label{P3}
Let $I=(u_1,\ldots,u_q,v)$ be a squarefree almost complete intersection ideal. Then $R/I$ is clean. In particular, $R/I$ is pretty clean. 
\end{Proposition}

\begin{proof}
Since $I$ is almost complete intersection, by using Lemma \ref{L1} we deduce that $\supp(u_i)\cap\supp(u_j)=\emptyset$ for all $1\leq i\neq j\leq q$. Thus $I^{\vee}=\bigcap_{i=1}^q\frak{p}_i\cap\frak{p}$, where
$\frak{p}_i=(x:~ x\in\supp(u_i))$ and $\frak{p}=(x:~ x\in\supp(v))$. It therefore follows that $I^{\vee}=J\cap\frak{p}$, where $J=\prod_{i=1}^{q}\frak{p_i}$. It is clear that $I^{\vee}$ is weakly polymatroidal and so it is linear quotients. Hence $R/I$ is clean, as desired.
\end{proof}

In \cite[Theorem 3.10]{S2} proved that if $I$ is monomial ideal, then $R/I$ is pretty clean if and only if $R/I^p$ is clean, where $I^p$ is a polarization of $I$. For more details on polarization, see  \cite{F} or \cite[Section 1.6]{HH}. Moreover, it is clear that $I$ is an almost complete intersection if and only if its polarization is also an almost complete intersection. Combining these insights with Proposition \ref{P3}, we obtain the following theorem which appears in \cite[Theorem 2.5]{BDS}.   

\begin{Proposition}\label{P2}
Let  $I=(u_1,\ldots,u_q,v)$ be almost complete intersection. Then $R/I$ is pretty clean. In particular, $R/I$ is sequentially Cohen-Macaulay.
\end{Proposition}

The following example demonstrates that the cleanness of almost complete intersection monomial ideals does not generally hold.
\begin{Example}
Suppose $I=(x^4,y^3z^2,x^2y^4z)$. Then $I$ is almost complete intersection and $\Ass(I)\neq\min\Ass(I)$. Thus $I$ is not clean.
\end{Example}

\begin{Corollary}
Let $I=(u_1,\ldots,u_q,v)$ be an almost complete intersection. Then the following conditions are equivalent:
\begin{enumerate}
	
	\item[(i)] $R/I$ is Cohen-Macaulay.
	\item[(ii)]  $I$ is unmixed.
	\item[(iii)] $R/I$ is clean. 
\end{enumerate}	
\end{Corollary}

\begin{proof}
$(i)\Longrightarrow (ii)$. It is clear.\\
$(ii)\Longrightarrow (iii)$. Since $I$ is unmixed, we have $\Ass(I)=\min\Ass(I)$ and from Proposition \ref{P2} and \cite[Corollary 3.4]{HP} we obtain $R/I$ is clean.\\
$(iii)\Longrightarrow (i)$. Since  $R/I$ is clean, we have $\Ass(I)=\min\Ass(I)$. Since $I$ is almost complete intersection, it therefore follows that $I$ is unmixed. Also, by applying Proposition \ref{P2} we deduce that $R/I$ is sequentially Cohen-Macaulay. Hence from \cite[Corollary 3]{CDSS}, we conclude that $R/I$ is Cohen-Macaulay, as desired.  
\end{proof}

\begin{Proposition}\label{P0}
Let $I=(u_1,\ldots,u_q,v)$ be an almost complete intersection ideal. Then $R/I$ is Cohen-Macaulay if and only if $I$ is non-dominant. 
\end{Proposition}

\begin{proof}
Since $I$ is almost complete intersection, it follows that $\height(I)=\mid G(I)\mid-1=q$ and also $q\leq\pd(R/I)=p\leq q+1$. By applying \cite[Corollary 4.1]{BR}, we have $\beta_i(I)\geq{p\choose {i+1}}$ for all $i=0,1,\ldots,p-1$. Hence, if $I$ is non-dominant, then by \cite[Corollary 4.9]{A} $p=q$ and so $\height(I)=\pd(R/I)$. Therefore, $R/I$ is Cohen-Macaulay. Conversely, if $I$ is Cohen-Macaulay then $p=\height(I)=\pd(R/I)$ and so $I$ is non-dominant, as desired.
\end{proof}

\begin{Example}
Suppose $I=(x^2,y^3,xy^2z)$. Then $I$ is an almost complete intersection such that $R/I$ is non Cohen-Macaulay, because of $I$ is a dominant ideal.
\end{Example}

\begin{Proposition}
Let $I$ be a dominant ideal. Then the following conditions are equivalent:
\begin{enumerate}
\item[(i)] $R/I$ is Cohen-Macaulay.
\item[(ii)]  $I$ is unmixed.
\item[(iii)] $I$ is complete intersection. 
\end{enumerate}	
\end{Proposition}

\begin{proof}
The implications $(i)\Longrightarrow (ii)$ and $(iii)\Longrightarrow (i)$ are straightforward. Therefore, the main task is to prove $(ii)\Longrightarrow (iii)$. To this end, note that since 
$I$ is dominant and unmixed, its polarization $I^p$ inherits these properties, that is, it is also dominant and unmixed. Consequently, by applying \cite{F}, we have  $\height(I)=\height(I^p)=\mid G(I^p)\mid=\mid G(I)\mid$.  Hence $I$ is a complete intersection ideal, as required.
\end{proof}

Suppose $I$ is a monomial ideal with minimal generating set $G(I)=\{u_1,\ldots,u_q\}$. The {\it Scarf simplicial complex}, denoted $\Delta_I$, is the collection of all subsets $\sigma\subseteq\{1,\ldots,q\}$ for which the least common multiple of the corresponding generators $u_{\sigma}=\lcm\{u_i: i\in\sigma\}$ is unique:
$\Delta_{I}=\{\sigma\subseteq\{1,\ldots,q\}: u_{\sigma}\neq u_{\tau}$ for all $\sigma\neq\tau\}$, see for details \cite{BPS}.

For the differential maps, we follow the construction analogous to Taylor's resolution: $d(e_{\sigma})=\Sigma_{i\in\sigma}sign(i,\sigma)\frac{u_{\sigma}}{u_{\sigma\setminus\{i\}}}e_{\sigma\setminus\{i\}}$, where $sing(i,\sigma)=(-1)^{j+1}$ if $i$ is the $j$th element in the ordering set $\sigma=\{i_1,\ldots,i_r\}$, and $e_{\sigma}=e_{i_1}\wedge e_{i_2}\wedge\ldots\wedge e_{i_r}$.

\begin{Proposition}\label{P1}
Let $I=\left(u_{1}, \ldots, u_{q}, v\right)$ be an almost complete intersection monomial ideal such that $v \mid \operatorname{lcm}\left(u_{1}, u_{2}\right)$. Then, the Betti numbers are given by $\beta_{i}\left(\frac{R}{I}\right)=\binom{q+1}{i}$ for $i=0,1$ and 
$$\beta_{i}\left(\frac{R}{I}\right)=\binom{q+1}{q+1-i}-\binom{q-1}{q+1-i},$$ for $i=2,3, \ldots, q$.
\end{Proposition}

\begin{proof}
Since $I=\left(u_{1}, \ldots, u_{q}, v\right)$ is an almost complete intersection monomial ideal with $v \mid \operatorname{lcm}\left(u_{1}, u_{2}\right)$, it follows that $I$ is a semidominant monomial ideal.  It is evident that all the elements of the bases with homological degrees 0 and 1 appear in the minimal free resolution of $\frac{R}{I}$. Consequently, we find that $\beta_{i}\left(\frac{R}{I}\right)=\binom{q+1}{i}$ for $i=0,1$. Now, we will determine the Betti numbers for $i=2, \ldots, q$. Since $I$ is semi-dominant, the minimal free resolution fulfilled by Scarf complex and now by applying \cite[Corollary 6.7]{A} we deduce that $\beta_{i}\left(\frac{R}{I}\right)=\# L_{i}+\# L_{i-1}$, where $L_{i}=\left\{\left\{j_{1}, \ldots, j_{i}\right\} \subseteq\{1, \ldots, q\} ; v \nmid\right.$ $\left.\operatorname{lcm}\left(u_{j_{1}}, \ldots, u_{j_{i}}\right)\right\}$. To find $\beta_{2}\left(\frac{R}{I}\right)$, we note that since $v \nmid u_{j}$, it follows that $\# L_{1}=q=\binom{q}{1}$. Moreover,  $v$ divides  $\operatorname{lcm}\left(u_{1}, u_{2}\right)$ but does not divide $\operatorname{lcm}\left(u_{r}, u_{s}\right)$ for each pair $r \neq 1,2$ or $s \neq 1,2$. Consequently, the only basis elements removed from the minimal free
resolution of $\frac{R}{I}$ in homological degree 2 is $e_{1} \wedge e_{2}$. Therefore $\# L_{2}=\binom{q}{2}-1=\binom{q}{2}-\binom{q-2}{0}$ and so $\beta_{2}\left(\frac{R}{I}\right)=\# L_{2}+\# L_{1}=\binom{q}{2}+q-1=\binom{q+1}{q-1}-\binom{q-1}{q-1}$. To find $\# L_{3}$, we observe that since $v \mid l c m\left(u_{1}, u_{2}\right)$, it follows that $v$ divides all basis that include $\left\{u_{1}, u_{2}\right\}$. Thus, we have $\# L_{3}=\binom{q}{3}-\binom{q-2}{1}$. Next, we deduce that $\beta_{3}\left(\frac{R}{I}\right)=\# L_{3}+$ $\# L_{2}=\binom{q}{3}-\binom{q-2}{1}+\binom{q}{2}-\binom{q-2}{0}=\binom{q+1}{q-2}-\binom{q-1}{q-2}$. Now, to determine $\# L_{j}$ for $3<j \leq q$, we again use the fact that $v \mid l c m\left(u_{1}, u_{2}\right)$, which means $v$ divides all basis containing $\left\{u_{1}, u_{2}\right\}$. Consequently, we have $\# L_{j}=\binom{q}{j}-\binom{q-2}{j-2}$. Hence, we have
$\beta_{j}\left(\frac{R}{I}\right)=\# L_{j}+\# L_{j-1}
=\binom{q}{j}-\binom{q-2}{j-2}+\binom{q}{j-1}-\binom{q-2}{j-3}.$
Hence, we easily obtain that $\beta_{i}\left(\frac{R}{I}\right)=\binom{q+1}{q+1-i}-\binom{q-1}{q+1-i},$ for $i=2,3, \ldots, q$,
as required.
\end{proof}

In the following result, we will prove Proposition \ref{P1} in general.   
\begin{Theorem}\label{T1}
 Let $I=\left(u_{1}, \ldots, u_{q}, v\right)$ be an almost complete intersection monomial ideal and let $s$ be the smallest positive integer number such that $v \mid l c m\left(u_{j_{1}}, \ldots, u_{j_{s}}\right)$, for some subset $\left\{j_{1}, \ldots, j_{s}\right\} \subseteq\{1, \ldots, q\}$. Then, for $i=0,1, \ldots, s-1$, we have $\beta_{i}\left(\frac{R}{I}\right)=\binom{q+1}{i}$ and for $i=s, s+1, \ldots, q$, it holds that $\beta_{i}\left(\frac{R}{I}\right)=\binom{q+1}{q+1-i}-\binom{q+1-s}{q+1-i}$.	
\end{Theorem}

\begin{proof}
Since $I=\left(u_{1}, \ldots, u_{q}, v\right)$ is an almost complete intersection monomial ideal with $v \mid \operatorname{lcm}\left(u_{j_{1}}, \ldots, u_{j_{s}}\right)$, for $\left\{j_{1}, \ldots, j_{s}\right\} \subseteq\{1, \ldots, q\}$ and $I$ is a semidominant monomial ideal, the \cite[Corollary 6.7]{A} implies that $\beta_{i}\left(\frac{R}{I}\right)=\# L_{i}+\# L_{i-1}$, where $L_{i}$ is defined as before. For $i=0,1, \ldots, s-1$, it is evident that all basis elements of homological degrees $0,1, \ldots, s-1$ appear in the minimal free resolution of $\frac{R}{I}$. Specifically, we have 
$\# L_{i}=\binom{q}{i}$, which yields $\beta_{i}\left(\frac{R}{I}\right)=\# L_{i}+$ $\# L_{i-1}=\binom{q}{i}+\binom{q}{i-1}=\binom{q+1}{i}$, for $i=0,1, \ldots, s-1$. For $i=s$, since $v \mid l c m\left(u_{j_{1}}, \ldots, u_{j_{s}}\right)$ and $v \nmid l c m\left(u_{l_{1}}, \ldots, u_{l_{s}}\right)$ if there exist $l_{k}$ such that $l_{k} \neq j_{t}$ for $t=1, \ldots, s$, and so the only basis element removed from the minimal free resolution of $\frac{R}{I}$ in homological degree $s$ is $e_{j_{1}} \wedge \ldots \wedge e_{j_{s}}$. Hence we deduce that $\# L_{s}=\binom{q}{s}-1=\binom{q}{s}-\binom{q-s}{0}$. To find $\# L_{s+1}$, since $v \mid l c m\left(u_{j_{1}}, \ldots, u_{j_{s}}\right)$, it follows that $v$ divides all the basis containing $\left\{u_{j_{1}}, \ldots, u_{j_{s}}\right\}$. Therefore, we obtain $\# L_{s+1}=\binom{q}{s+1}-\binom{q-s}{1}$. To find $\# L_{j}$ for $s+1<j \leq$ $q$, since $v \mid \operatorname{lcm}\left(u_{j_{1}}, \ldots, u_{j_{s}}\right)$, it follows that $v$ divides all the basis containing $\left\{u_{j_{1}}, \ldots, u_{j_{s}}\right\}$. Then $\# L_{j}=$ $\binom{q}{j}-\binom{q-s}{j-s}$. Hence

$$
\begin{aligned}
	\beta_{i}\left(\frac{R}{I}\right)&=\# L_{i} +\# L_{i-1}\\
	&=\binom{q}{i}-\binom{q-s}{i-s}+\binom{q}{i-1}-\binom{q-s}{i-s-1} \\
	& =\binom{q+1}{i}-\binom{q+1-s}{i-s}.
\end{aligned}
$$

Consequently, we have $\beta_{i}\left(\frac{R}{I}\right)=\binom{q+1}{q+1-i}-\binom{q+1-s}{q+1-i}$ for $i=s, s+1, \ldots, q$, as desired.
\end{proof}

The following example has been appeared in \cite[Example 1.6]{HHMY}.

 \begin{Example}\label{E2}
Let $2\leq j<q-1$ and consider the ideal \\ $I=(x_1y_1,x_2y_2,\ldots,x_{q}y_{q},y_1\ldots y_j)$ of $K[x_1,\ldots,x_q,y_1,\ldots,y_q]$. 
 Applying Theorem \ref{T1}, we deduce that $\beta_{i}(I)=\binom{q+1}{i+1}$ for $i=0,1, \ldots, j-2$ and $\beta_{i}(I)=\binom{q+1}{q-i}-\binom{q+1-j}{q-i}$ for $i=j-1,\ldots,q-1$.	
 \end{Example}

\begin{Example}\label{E3}
	
Consider an almost complete intersection ideal\\ $I=(x_1^{a_1},x_2^{a_2},\ldots,x_n^{a_n},x_1^{b_1}\ldots x_n^{b_n})$, where $a_i>b_i>0$ for all $ 0\leq i\leq n$.
Then the Betti numbers are  $\beta_i(I)=\binom{n+1}{i+1}$ for $i=0,1,\ldots, n-2$ and  $\beta_{n-1}(I)=n$.

\end{Example}

\begin{Theorem}\label{T2}
Let $I$ be an almost complete intersection squarefree monomial ideal. Then the following statements hold:
\begin{enumerate}
	
\item[(a)] if $I$ satisfies in the conditions in case $(i)$ of Theorem \ref{T0}, then for $i=0,1, \ldots, p-1$, we have $\beta_{i}\left(\frac{R}{I}\right)=\binom{q+1}{i}$ and
it holds that
$\beta_{i}\left(\frac{R}{I}\right)=\binom{q+1}{q+1-i}-\binom{q+1-p}{q+1-i}$, for $i=p, p+1, \ldots, q$.	
\item[(b)] if $I$ satisfies the conditions in cases $(ii)$ and $(vi)$ of Theorem \ref{T0}, then\\ $\beta_{i}\left(\frac{R}{I}\right)=\binom{q+1}{i}$ for $i=0,1,\ldots,q$.	
\item[(c)] if $I$ satisfies the conditions in cases $(iii)$,$(iv)$ and $(v)$ of Theorem \ref{T0}, then\\  $\beta_{i}\left(\frac{R}{I}\right)=\binom{q+1}{i}-\binom{q-1}{i-2}$
for $i=0,1,\ldots,q$. 
\end{enumerate}	
\end{Theorem}

\begin{proof}
Case $(a)$ is satisfied by Theorem \ref{T1}. In case $(b)$, since 
$I$ is a dominant ideal, it admits a minimal free resolution given by the Taylor resolution. Consequently, the Betti numbers are $\beta_{i}\left(\frac{R}{I}\right)=\binom{q+1}{i}$ for $i=0,1,\ldots,q$.\\
For case $(c)$, we assume that $I$ satisfies the condition $(iii)$ of Theorem \ref{T0}. The proofs for the other cases $(iv)$ and $(v)$ of Theorem \ref{T0} are similar. 
Since  $v_2v_3 \mid l c m\left(v_1v_2,v_1v_3\right)$, it follows that $v_2v_3$ divides only those subsets of generators that contain $v_1v_2,v_1v_3$. Hence, the division process does not occur in any other cases. By applying the argument from Theorem \ref{T1}, we find $\# L_{i}=\binom{q}{i}-\binom{q-2}{i-2}$. Therefore, the Betti numbers satisfy $\beta_{i}\left(\frac{R}{I}\right)=\# L_{i}+\# L_{i-1}=\binom{q}{i}-\binom{q-2}{i-2}+\binom{q}{i-1}-\binom{q-2}{i-3}$. Simplifying yields $\beta_{i}\left(\frac{R}{I}\right)=\binom{q+1}{i}-\binom{q-1}{i-2}$ for $i=0,1,\ldots,q$, as required.

\end{proof}

Since the Betti numbers and the property of being an almost complete intersection are preserved under polarization  (see for instance \cite[Corollary 1.6.3]{HH} or \cite{F}), we can now state the following result:

\begin{Theorem}\label{T3}
Let 
$I$ be an almost complete intersection monomial ideal. Then 
$I$ satisfies the conditions outlined in Theorem \ref{T2}.
\end{Theorem}

\begin{Example}
Let $I=(x_1^2x_2x_3^3,x_5x_2x_4^5,x_3^3x_4^5,x_6x_7^3,x_8x_9^2)$ be an almost complete intersection satisfying in case $(v)$ of Theorem \ref{T0}. Applying Theorem \ref{T2}$(c)$, we deduce the Betti numbers of $R/I$ as follows: $\beta_{0}\left(\frac{R}{I}\right)=1$, $\beta_{1}\left(\frac{R}{I}\right)=5$, $\beta_{2}\left(\frac{R}{I}\right)=9$, $\beta_{3}\left(\frac{R}{I}\right)=7$ and $\beta_{4}\left(\frac{R}{I}\right)=2$. 
\end{Example}

Let $I=(u_1,\ldots,u_q)$ be a complete intersection monomial ideal, and let $s$ be a positive integer. The Eagon-Northcott resolution \cite{EN} provides the Betti numbers of $I^s$ as follows:

$\beta_{i}(I^s)=\binom{q+s-1}{s+i}\binom{s+i-1}{i}$ for all $i=0,1,\ldots,q-1$, see also \cite{GV}.  Given this, a natural question arises:

\begin{Question}
Let $I=(u_1,\ldots,u_q,v)$ be an almost complete intersection monomial ideal and let $s$ be a positive integer. What is the formula for the Betti numbers of $I^s$?
\end{Question}




\end{document}